\documentclass[11pt]{amsart}
\usepackage{amsmath,amssymb,latexsym,cite}
\usepackage[small]{caption}
\usepackage{graphicx,color,mathrsfs,tikz}%wasysym,overpic,tikz,
\usepackage{subfigure,color}
\usepackage{cite}
\usepackage[colorlinks=true,urlcolor=blue,
citecolor=red,linkcolor=blue,linktocpage,pdfpagelabels,
bookmarksnumbered,bookmarksopen]{hyperref}
\usepackage[italian,english]{babel}
\usepackage{units}
\usepackage{enumitem}
\usepackage[left=2.7cm,right=2.7cm,top=2.9cm,bottom=2.9cm]{geometry}
\usepackage[hyperpageref]{backref}

\usepackage[colorinlistoftodos,prependcaption]{todonotes}
\numberwithin{equation}{section}
\newtheorem{theorem}{Theorem}[section]
\newtheorem{proposition}{Proposition}[section]
\newtheorem{lemma}{Lemma}[section]
\newtheorem{remark}{Remark}[section]

\newtheorem{corollary}{Corollary}[section]

\renewcommand{\epsilon}{\eps}
\renewcommand{\i}{{\rm i}}

\newcommand{\A}{{\mathcal A}}

\newcommand{\CC}{{\mathcal C}}
\newcommand{\C}{{\mathbb C}}
\newcommand{\N}{{\mathbb N}}
\newcommand{\R}{{\mathbb R}}

\newcommand{\eps}{\varepsilon}

\newcommand{\M}{\mathscr{M}}
\newcommand{\iu}{{\rm i}}

\newcommand{\pnorm}[2][]{\if #1'' \left|#2\right|_p \else \left|#2\right|_{#1} \fi}

\newcommand{\loc}{{\rm loc}}

\renewcommand{\theta}{\vartheta}

\DeclareMathOperator{\supp}{supp}
\title{New characterizations of magnetic Sobolev spaces}

\author[H.-M. Nguyen]{Hoai-Minh Nguyen}
\author[A.\ Pinamonti]{Andrea Pinamonti}
\author[M.\ Squassina]{Marco Squassina}
\author[E.\ Vecchi]{Eugenio Vecchi}
\address[H.-M. Nguyen]{Department of Mathematics \newline\indent
	EPFL SB CAMA \newline\indent
	Station 8 CH-1015 Lausanne, Switzerland}
\email{hoai-minh.nguyen@epfl.ch}

\address[A.\ Pinamonti]{Dipartimento di Matematica \newline\indent
	Universit\`a di Trento \newline\indent
	Via Sommarive 14, 38050 Povo (Trento), Italy}
\email{andrea.pinamonti@unitn.it}

\address[M.\ Squassina]{Dipartimento di Matematica e Fisica \newline\indent
	Universit\`a Cattolica del Sacro Cuore \newline\indent
	Via dei Musei 41, I-25121 Brescia, Italy}
\email{marco.squassina@unicatt.it}

\address[E.\ Vecchi]{Dipartimento di Matematica \newline\indent
	Universit\`a di Bologna \newline\indent
	Piazza di Porta S. Donato 5, 40126, Bologna, Italy}
\email{eugenio.vecchi2@unibo.it}

\thanks{A.P., M.S.\ and E.V.  are members of {\em Gruppo Nazionale per l'Analisi Ma\-te\-ma\-ti\-ca, la Probabilit\`a e le loro Applicazioni} (GNAMPA) of the {\em Istituto Nazionale di Alta Matematica} (INdAM). E.V. receives funding from the People Programme (Marie Curie Actions) of the European Union's Seventh
	Framework Programme FP7/2007-2013/ under REA grant agreement No.\ 607643 (ERC Grant MaNET `Metric Analysis for Emergent Technologies').}



\begin{document}
	%\hyphenation{Spia-na-to}
	\begin{abstract}
		We establish two new characterizations of magnetic Sobolev spaces  for Lipschitz magnetic fields in terms of nonlocal functionals. The first one is related to the BBM formula, due to Bourgain, Brezis, and Mironescu. The second one is related to the work of  the first author on the classical Sobolev spaces. We also study the convergence almost everywhere and the convergence in $L^1$ appearing  naturally in these contexts. 
		
		\end{abstract}
	
	\maketitle
	
%	\begin{center}
%		\begin{minipage}{8cm}
%			\small
%			\tableofcontents
%		\end{minipage}
%	\end{center}
%	\medskip
	
	\section{Introduction}
	In electromagnetism, a relevant role in the study of particles which interact 
	with a magnetic field $B=\nabla\times A$, $A:\R^3\to\R^3$,  is played by the magnetic Laplacian $(\nabla-\iu A)^2$ \cite{AHS,reed,LL}. This 
	yields to nonlinear Schr\"odinger equations of the type $- (\nabla-\iu A)^2 u + u = f(u),$ which have been extensively studied (see e.g.\ \cite{EstLio,arioliSz,mills,lauli} 
	and the references therein). The linear operator $- (\nabla-\iu A)^2 u$ is defined weakly as the differential of the energy functional
	$$
	H_{A}^{1}(\R^N)\ni u\mapsto \int_{\R^N}|\nabla u-\i A(x)u|^2dx,
	$$
	over complex-valued functions $u$ on $\R^N$. Here $\i$ denotes the imaginary unit and $|\cdot|$ the standard Euclidean norm of $\C^N$.   Given a  measurable function $A:\R^N\to\R^N$ and given  an open subset $\Omega$ of $\R^N$, one defines $H^{1}_A(\Omega)$ as the space of complex-valued functions $u\in L^2(\Omega)$ such that  $\|u\|_{H^1_A(\Omega)}<\infty$
	for the norm
	$$
	\|u\|_{H^{1}_A(\Omega)}:=\Big(\|u\|_{L^2(\Omega)}^2+[u]_{H^{1}_A
		(\Omega)}^2\Big)^{1/2},\quad 
	[u]_{H^{1}_A(\Omega)}:=\Big(\int_{\Omega}|\nabla u-\i A(x)u|^2dx\Big)^{1/2}.
	$$

	In \cite{I10}, some physically motivated  nonlocal versions of the 
	local magnetic energy were introduced. In particular the operator $(-\Delta)^s_A$ is defined as the gradient of the nonlocal energy functional
\begin{equation*}
	H^s_A(\R^N)\ni u\mapsto (1-s)\iint_{\R^{2N}}\frac{|u(x)-e^{\i (x-y)\cdot A\left(\frac{x+y}{2}\right)}u(y)|^2}{|x-y|^{N+2s}}dx \, dy,
\end{equation*}
	where $s\in (0,1)$. Recently, the existence of ground stated of $(-\Delta)^s_Au+u=f(u)$ was investigated in \cite{piemar} via Lions concentration compactness arguments.
	In \cite{BM} a  connection between the local and nonlocal notions was obtained on bounded domains, precisely, if $\Omega\subset\R^N$ is a bounded Lipschitz domain and
	$A\in C^2(\R^N)$, then for every $u\in H^1_{A}(\Omega)$ it holds
	\begin{equation}
	\label{limit-introduction}
	\lim_{s\nearrow 1}(1-s)\int_{\Omega}\int_{\Omega}\frac{|u(x)-e^{\i (x-y)\cdot A\left(\frac{x+y}{2}\right)}u(y)|^2}{|x-y|^{N+2s}}dx \, dy=
	Q_N\int_{\Omega}|\nabla u-\i A(x)u|^2dx,
	\end{equation}
	where 
	\begin{equation}
		\label{valoreK}
		Q_{N}:=\frac{1}{2}\int_{{\mathbb S}^{N-1}}|{\boldsymbol \omega}\cdot \sigma|^{2}d \sigma
	\end{equation}
	being ${\mathbb S}^{N-1}$ the unit sphere in $\R^N$
	and ${\boldsymbol \omega}$ an arbitrary unit vector of $\R^N$. 
	See also \cite{PSV} for the general case of the $p$-norm with $1 \le p < + \infty$
	as well as  \cite{PSV2} where the limit as $s\searrow 0$ is covered.
	This provides a new characterization of the $H^1_A$ norm in terms of nonlocal functionals extending the results by Bourgain, Brezis and Mironescu \cite{bourg,bourg2} (see also \cite{davila,ponce}) to the magnetic setting. 
	Let $\{s_n\}_{n\in\N}$ be a sequence of positive numbers converging to $1$ and less than $1$ and  set 
	\begin{equation*}
	%\label{hat-rho-introduction}
	 \rho_{n}(r) :=\left\{ \begin{array}{cl}
	2 (1 - s_n) \mbox{diam}(\Omega)^{2s_n -2}r^{2 -2s_n - N} &  \mbox{ for } 0 < r  \le  \mbox{diam} (\Omega), \\[4pt] 
	0 & \mbox{ for } r> \mbox{diam} (\Omega),  
	\end{array}\right. 
	\end{equation*}
	where $\mbox{diam} (\Omega)$ denotes the diameter of $\Omega$. 
	We have $\int_{0}^\infty \rho_n(r) r^{N-1}dr  = 1$ and, for all $\delta>0,$
	$$
	\lim_{n \to + \infty} \int_{\delta}^\infty \rho_n(r) r^{N-1} \, dr  = 0. 
	$$
	Given  $u: \Omega \to\C$ a measurable complex-valued function, we denote
	\begin{equation*}
		\label{def-Psi}
	\Psi_u(x,y):=e^{\i (x-y)\cdot A\left(\frac{x+y}{2}\right)}u(y),\quad\,\, x,y\in \Omega. 
	\end{equation*}
The function $\Psi_u(\cdot, \cdot)$ also depends on $A$ but for notational ease, we ignore it. Assertion \eqref{limit-introduction} can be then written as 
%	\begin{equation}\label{limit-introduction-1}
%	\textcolor{red}{\lim_{n \to  + \infty} \int_{\Omega}\int_{\Omega}\frac{|\Psi_u(x, y) - \Psi_u(x, x)|^2}{|x-y|^{2}} \hat \rho_n (|x -y|) \, dx \, dy=
%	2 Q_N\int_{\Omega}|\nabla u-\i A(x)u|^2dx.} 
%	\end{equation}
	\begin{equation}\label{limit-introduction-1}
	\lim_{n \to  + \infty} \int_{\Omega}\int_{\Omega}\frac{|\Psi_u(x, y) - \Psi_u(x, x)|^2}{|x-y|^{2}}  \rho_n (|x -y|) \, dx \, dy=
	2 Q_N\int_{\Omega}|\nabla u-\i A(x)u|^2dx.
	\end{equation}

%We stress that $H^{1}_A(\R^N)$ is also the closure of the space $C^\infty_c(\R^N)$ in $H^{1}_A(\R^N)$, see e.g.,  \cite[Proposition 2]{EstLio}.

This paper is concerned with the {\em whole space} setting. Our  first goal is to obtain formula \eqref{limit-introduction-1} for $\Omega=\R^N$
and to provide a  characterization of $H^1_A(\R^N)$ in terms of the LHS of \eqref{limit-introduction-1} in the spirit of the work of Bourgain, Brezis and Mironescu.  

\medskip 
\noindent
Here and in what follows, a sequence of nonnegative radial functions $\{\rho_n\}_{n\in\N}$ 
is called a {\em sequence of mollifiers} if it satisfies the conditions
\begin{equation}
\label{cond-kern}
	\int_{0}^{\infty} \rho_n(r) r^{N-1} dr=1 \quad \mbox{ and } \quad 		\lim_{n\to +  \infty}\int_{\delta}^{\infty} \rho_n(r) r^{N-1} dr=0,\,\,\, \mbox{for all $\delta > 0$}. 
\end{equation} 

In this direction, we have 
\begin{theorem} 
	\label{main-BBM} 
	Let  $A:\R^N\to\R^N$ be Lipschitz  and let $\{\rho_n\}_{n\in\N}$ be a sequence of nonnegative radial mollifiers. 
Then $u \in H^1_A(\R^N)$ if and only if $u \in L^2(\R^N)$ and 
\begin{equation}\label{main-BBM-1}
\sup_{n \in \N}  \iint_{\R^{2N}} \frac{|\Psi_u (x, y) -  \Psi_u(x, x)|^2}{|x - y|^2} \rho_n (|x - y|) \, dx \, dy < + \infty. 
\end{equation}
Moreover, for $u \in H^1_A(\R^N)$, we have
\begin{equation}\label{main-BBM-2}
\lim_{n \to + \infty} \iint_{\R^{2N}} \frac{|\Psi_u (x, y) -  \Psi_u(x, x)|^2}{|x - y|^2} \rho_n (|x - y|) \, dx \, dy = 2 Q_N  \int_{\R^N} |\nabla u - \i A(x) u|^2 \, dx,
\end{equation}
and 
\begin{multline}\label{main-BBM-3}
\iint_{\R^{2N}} \frac{|\Psi_u (x, y) -  \Psi_u(x, x)|^2}{|x - y|^2} \rho_n (|x - y|) \, dx \, dy \\
\le 2|\mathbb{S}^{N-1}| \int_{\R^N} |\nabla u - \i A(x) u|^2 \, dx + 2|\mathbb{S}^{N-1}| \big(2+\|\nabla A \|_{L^\infty(\R^N)}^2 \big)  \int_{\R^N} |u|^2 \, dx. 
\end{multline}
\end{theorem}

In this paper, $|\mathbb{S}^{N-1}|$ denotes the $(N-1)$-Hausdorff measure of the unit sphere $\mathbb{S}^{N-1}$ in $\R^N$. 

\medskip
The proof of Theorem~\ref{main-BBM} is given in Section~\ref{sect-BBM}.

\begin{remark} \rm 
	%Let $\{s_n\}_{n\in\N}$ be a sequence of positive numbers converging to $1$ and less than $1$. 
	Similar results as in Theorem~\ref{main-BBM} hold for more general mollifiers $\{ \rho_n\}_{n \in \N}$ with slight changes in the constants.  See Remark~\ref{rem-Hs} for details.
%This in particular yields  \eqref{limit-introduction} with $\Omega = \R^N$. 
\end{remark}
			
The second goal of this paper is to characterize $H^1_A(\R^N)$ in term of $J_\delta (\cdot )$ where, for $\delta > 0$,  
\begin{equation*}
J_\delta(u): = \iint_{\{|\Psi_u(x,y)-\Psi_u(x,x)|>\delta\}}\frac{\delta^2}{|x-y|^{N+2}}\, dxdy,\quad \mbox{for $u \in L^1_{{\rm loc}} (\R^N)$}. 
\end{equation*}
This is motivated by the characterization of the Sobolev space $H^1(\R^N)$ provided in \cite{BourNg, nguyen06} (see also \cite{bre,bre-linc,BHN,BHN2,BHN3, NgSob2, NgGamma, nguyen11, NgSob4})
	in terms of the family of nonlocal functionals $I_\delta$ which is defined by, for $\delta > 0$, 
\begin{equation*}
	I_\delta(u):=\iint_{\{|u(y)-u(x)|>\delta\}}\frac{\delta^2}{|x-y|^{N+2}}dx \, dy,\quad  \mbox{ for } u \in L^1_{{\rm loc}}(\R^N). 
\end{equation*}
It was showed in \cite{BourNg, nguyen06} that if $u\in L^2(\R^N)$,
	then $u\in H^1(\R^N)$ if and only if $\sup_{0<\delta<1} I_\delta(u)<\infty$; moreover, 
	\begin{equation*}
		%\label{newcar}
		\lim_{\delta\searrow 0} I_\delta(u)=Q_N\int_{\Omega}|\nabla u|^2dx, \quad \mbox{for $u \in H^1(\R^N)$}. 
	\end{equation*}
	 Concerning this direction, we establish 
	
	\begin{theorem}
		\label{main} Let $A:\R^N\to\R^N$ be Lipschitz.  Then $u\in H^1_A(\R^N)$ if and only if $u \in L^2(\R^N)$ and 
		\begin{equation}\label{bound}
		\sup_{0<\delta<1} J_\delta(u) < + \infty.
		\end{equation}
		Moreover, we have, for $u \in H^1_A(\R^N)$, 
		\begin{equation*}
		\lim_{\delta\searrow 0} J_\delta(u) = 	Q_{N}\int_{\R^N}|\nabla u-\i A(x)u|^2\, dx
		\end{equation*}
		and 
		\begin{equation}\label{main-estimate}
		\sup_{\delta > 0}J_\delta (u)
\leq C_N  \left( \int_{\R^N}|\nabla u-\i A(x)u|^2\, dx + \big(\|\nabla A\|_{L^\infty(\R^N)}^2 + 1 \big) \int_{\R^N} |u|^2  \, dx \right). 
		\end{equation}
	\end{theorem}

	\medskip 
Throughout the paper, we shall denote by  $C_N$ a generic positive constant depending only on $N$ and possibly changing from line to line.

	\medskip 
	The proof of Theorem~\ref{main} is given in Section~\ref{sect-H1A}. 

\medskip 
As pointed out in \cite{EstLio}, a physically meaning example of magnetic potential in the space is
$$
A(x,y,z)=\frac{1}{2}(-y,x,0),\quad (x,y,z)\in\R^3,
$$
which in fact fulfills the requirement of Theorems~\ref{main-BBM} and \ref{main} that $A$ is Lipschitz.
Furthermore, in the spirit of \cite{bre}, as a byproduct of Theorems~\ref{main-BBM} and \ref{main},  for $u \in L^2(\R^N)$,  if we have
	$$
	\lim_{n \to +\infty}  \iint_{\R^{2N}} \frac{|\Psi_u (x, y) -  \Psi_u(x, x)|^2}{|x - y|^2} \rho_n (|x - y|) \, dx \, dy= 0 
	$$
	or
	$$
	\lim_{\delta\searrow 0} J_\delta(u) =0,
	$$
	then 
	\begin{equation*}
		\begin{cases}
			\nabla \Re u=- A \Im u, &\\
			\nabla \Im u= A\Re u,
		\end{cases}
	\end{equation*}
	namely the direction of $\nabla \Re u,\nabla \Im u$ is that of the 
	magnetic potential $A$.\ In the particular case $A=0$, this implies that $u$ is a constant function.  
	
	\vskip3pt
	\noindent
%	\textcolor{blue}{Although the classical magnetic Sobolev spaces are typically used in the Hilbert case $H^1_A(\R^N)$, for the sake of completeness we mention that the assertions
%	of Theorem~\ref{main} can be extended, with minor modifications, to the general case of the spaces $W^{1,p}_A(\R^N)$ (with obvious adaptations in the above definitions). Consider the space 
%	$(\mathbb{C}^N, |\cdot|_{p})$,  with the norm
%	\begin{equation*}
%	|z|_p:=\left(|(\Re z_1,\ldots, \Re z_N)|^p+|(\Im z_1,\ldots, \Im z_N)|^p\right)^{1/p},
%	\end{equation*}
%	where $|\cdot|$ is the Euclidean norm of $\R^N$ and
%	$\Re a$,$\Im a$ denote the real and imaginary parts of $a\in\mathbb{C}$
%	respectively. Notice that $|z|_p=|z|$ whenever $z 	\in \R^n$, which makes our next 
%	statements consistent with the case $A=0$ and $u$ being 
%	a real valued function. Also $|\cdot|_2=|\cdot|,$ consistently with the previous definition. We stress
%	that the $|\cdot|_p$-norm just defined should not be confused with the $p-norm$ on $\R^N$.}
The $L^p$ versions of the above mentioned results are  given in Sections~\ref{sect-BBM} and \ref{sect-H1A}. 
In addition to these results, we also discuss the convergence almost everywhere and the convergence in $L^1$ of  
the quantities appearing in Theorems~\ref{main-BBM}  and \ref{main} in Section~\ref{sect-pointwise}.

% In addition to the assertions
%of Theorem~\ref{main}, under the same assumptions, we also prove (Proposition~ \ref{pro-main} in Section~\ref{sect-H1A}) that
%\begin{equation*}
%\lim_{\delta \searrow 0} J_\delta(u, \cdot) = Q_N|\nabla u( \cdot ) - \i A(\cdot ) u(\cdot )|^2\quad \mbox{in $L^1(\R^N)$},   
%\end{equation*}
%where 
%$$
%J_\delta(u, x) :=  \int_{\{|\Psi_u(x,y)-\Psi_u(x,x)|>\delta\}}\frac{\delta^2}{|x-y|^{N+2}}dy. 
%$$

\medskip The paper is organized as follows. The proof of Theorems~\ref{main-BBM} and \ref{main} are given  in Sections~\ref{sect-BBM} and ~\ref{sect-H1A} respectively. The convergence almost everywhere and the convergence in $L^1$ are investigated in Section~\ref{sect-pointwise}.

\section{Proof of Theorem~\ref{main-BBM} and its $L^p$ version} \label{sect-BBM}

The proof of Theorem~\ref{main-BBM} can be derived from a few lemmas which we present below. The first one is on \eqref{main-BBM-3}.

\begin{lemma} [Upper bound]\label{lem-L1}
Let $A: \R^N \to \R^N$ be Lipschitz  and let $\{\rho_n\}_{n\in\N}$ be a sequence of nonnegative radial mollifiers. We have, for all $u \in H^1_A(\R^N)$, 
\begin{multline*}
%\label{L1}
\iint_{\R^{2N}} \frac{|\Psi_u (x, y) -  \Psi_u(x, x)|^2}{|x - y|^2} \rho_n (|x - y|) \, dx \, dy \\
\le 2  |\mathbb{S}^{N-1}| \int_{\R^N} |\nabla u - \i A(x) u|^2 \, dx + 2 |\mathbb{S}^{N-1}| \big(2+ \|\nabla A \|_{L^\infty(\R^N)}^2 \big)  \int_{\R^N} |u|^2 \, dx. 
\end{multline*}
\end{lemma}

\begin{proof} Since $C^\infty_c(\R^N)$ is dense in $H^1_A(\R^N)$ (cf.\ \cite[Theorem 7.22]{LL}), using Fatou's lemma, without loss of generality, one might assume that $u \in C^1_c(\R^N)$. Recall that
\begin{equation}\label{integral-rho-n}
	\int_{\R^N}\rho_n(|z|)dz=|\mathbb{S}^{N-1}|\int_0^\infty \rho_n(r)r^{N-1} \, dr =|\mathbb{S}^{N-1}|. 
\end{equation}
	 Since 
\begin{multline*}
\mathop{\iint_{\R^{2N}}}_{\{|x-y| \ge 1\}} \frac{|\Psi_u (x, y) -  \Psi_u(x, x)|^2}{|x - y|^2} \rho_n (|x - y|) \, dx \, dy \\[6pt]
\le 2 \iint_{\R^{2N}} \big(|u(y)|^2 + |u(x)|^2\big) \rho_n (|x - y|) \, dx \, dy  \leq 4 |\mathbb{S}^{N-1}| \int_{\R^N} |u|^2 \, dx, 
\end{multline*}
it suffices to prove that 
\begin{multline}\label{L1-part1}
\mathop{\iint_{\R^{2N}}}_{\{|x-y| < 1 \}} \frac{|\Psi_u (x, y) -  \Psi_u(x, x)|^2}{|x - y|^2} \rho_n (|x - y|) \, dx \, dy \\
\le 2  |\mathbb{S}^{N-1}| \left(\int_{\R^N} |\nabla u - \i A(x) u|^2 \, dx + \|\nabla A \|_{L^\infty(\R^N)}^2 \int_{\R^N} |u|^2 \, dx \right). 
\end{multline}
For a.e. $x,y \in \R^N$, we have
%\footnote{The computation in \eqref{DPsi} is not completely rigorous for $A$ which is Lipschitz. However, we only use estimates of the type \eqref{L1-part1-2}  below which do hold for Lipschitz $A$. For simple representation, we ignore the details for the rigor of this point.}
\begin{align*}
%\label{DPsi}
\frac{\partial \Psi_u (x, y)}{\partial y}  &= e^{\i (x - y)\cdot A \left( \frac{x + y}{2} \right)} \nabla u(y)- \i \left\{ A\Big( \frac{x + y}{2} \Big) + \frac{1}{2} (y-x) \cdot \nabla A \Big(\frac{x + y}{2} \Big)\right\} \times \\
& \times e^{\i (x - y)\cdot A \left( \frac{x + y}{2} \right)} u(y).  \notag
\end{align*}
It follows that 
\begin{equation}\label{p1-0}
\Big|\frac{\partial \Psi_u(x, y)}{\partial y} \Big| \le |\nabla u(y) - \i A(y) u(y)| + \Big| A\Big( \frac{x + y}{2} \Big)  - A(y)  \Big| |u(y)|+ \frac{1}{2} |y - x| \Big|\nabla A \Big( \frac{x + y}{2}\Big) \Big| |u(y)|.  
\end{equation}
This implies 
\begin{equation*}
\Big| \frac{\partial \Psi_u(x, y)}{\partial y}  \Big| \le |\nabla u(y) - \i A(y) u(y)| + \|\nabla A \|_{L^\infty(\R^N)} |x-y| |u(y)|, 
\end{equation*}
which yields, for $x, y \in \R^N$ with $|x-y| < 1$,  
\begin{align}\label{L1-part1-2}
\frac{|\Psi_u(x, y) - \Psi_u(x, x)|^2}{|x - y|^2} \le &  2 \int_0^1 \big|\nabla u \big(t y + (1-t)x \big)   - \i A \big(t y + (1-t)x \big)   u \big(t y + (1-t)x \big)  \big|^2 \, dt   \nonumber\\[6pt] 
& + 2 \|\nabla A \|_{L^\infty(\R^N)}^2  \int_0^1 \big|u \big(t y + (1-t)x \big) \big|^2 \, dt.  
\end{align}
Since, for $f \in L^2(\R^N)$, in light of \eqref{cond-kern} and \eqref{integral-rho-n},  we get
\begin{multline*}
\int_{\R^N} \int_{\R^N}  \int_0^1 \big|f\big(t y + (1 - t) x \big) \big|^2  \rho_n(|x -y|)  \, dt \, dx \, dy   \\[6pt]
= \int_{\R^N} |f(x)|^2 \, dx \int_{\R^N} \rho_n(|z|) \, dz = |\mathbb{S}^{N-1}|\int_{\R^N} |f(x)|^2 \, dx, 
\end{multline*}
we then derive from \eqref{L1-part1-2} that 
\begin{multline*}
\mathop{\iint_{\R^{2N}}}_{\{|x-y| < 1 \}} \frac{|\Psi_u (x, y) -  \Psi_u(x, x)|^2}{|x - y|^2} \rho_n (|x - y|) \, dx \, dy \\[6pt]
\le 2  |\mathbb{S}^{N-1}| \int_{\R^N} |\nabla u(y) - \i A(y) u(y)|^2  \, dy + 2 |\mathbb{S}^{N-1}| \|\nabla A \|_{L^\infty(\R^N)}^2 \int_{\R^N } |u(y)|^2 \, dy,   
\end{multline*}
which is \eqref{L1-part1}.
\end{proof}

We next establish the following result which is used in the proof of \eqref{main-BBM-2} and in the proof of Theorem~\ref{main}. 

\begin{lemma} \label{lem-L2} Let $u \in C^2(\R^N)$, $A: \R^N \to \R^N$ be Lipschitz,  
	and let $\{\rho_n\}_{n\in\N}$ be a sequence of nonnegative radial mollifiers.  Then 
	\begin{equation}
	\label{stimasit}
		\liminf_{n \to +\infty } \iint_{\R^{2N}} \frac{|\Psi_u (x, y) -  \Psi_u(x, x)|^2}{|x - y|^2} \rho_n (|x - y|) \, dx \, dy \ge 2 Q_N \int_{\R^N} |\nabla u - \i A(x) u|^2 \, dx. 
	\end{equation}
	 Moreover, for any $(\varepsilon_n)  \searrow 0$,  there holds
	\begin{equation}\label{stimasit2}
		\liminf_{n \to + \infty } \iint_{\R^{2N}} \frac{|\Psi_u (x, y) -  \Psi_u(x, x)|^{2+\varepsilon_n}}{|x - y|^{2+\varepsilon_n}}  \rho_n(|x-y|)\, dx \, dy \ge 2 Q_N \int_{\R^N} |\nabla u - \i A(x) u|^2 \, dx.
	\end{equation}
\end{lemma}

Throughout this paper, for $R>0$, let $B_R$ denote the open ball in $\R^N$ centered at the origin and of radius $R$.

\begin{proof} Fix $R>1$ (arbitrary).  
Using the fact
$$
|e^{\i t} - (1 + \i t)| \le C t^2, \quad\mbox{for $t \in \R$}, 
$$
we have, for $x, y \in B_R$, 
\begin{multline*}
\left| \Psi_u(x, y) - \Big(1 + \i (x - y) \cdot A (y)\Big) u (y)  \right|  \le
\left| \Psi_u(x, y) - \left(1 + \i (x - y) \cdot A \Big(\frac{x+y}{2} \Big) \right) u (y)  \right| \\[6pt]
 + |x - y| \Big| A \Big(\frac{x+y}{2} \Big)  - A(y) \Big| |u(y)|  \le C \| u \|_{C^2(B_R)} (1 +  \| A\|_{W^{1, \infty} (B_R)} )^2  |x-y|^2. 
\end{multline*}
Here and in what follows, $C$ denotes a positive constant.  On the other hand, we obtain, for $x, y \in B_R$,  
\begin{equation*}
\big|u (x) - u (y)  - \nabla u(y) \cdot (x -y)\big| \le C \| u\|_{C^2(B_R)} |x-y|^2. 
\end{equation*}
It follows that  
\begin{multline}\label{L2-part1}
\left| \Big[\Psi_u(x, y) - \Psi_u(x, x) \Big] - \Big(\nabla u(y) -  \i  A  (y) u(y )  \Big) \cdot (y-x) \right| \\[6pt]
\le  C \| u\|_{C^2(B_R)} \big(1 + \| A\|_{W^{1, \infty}(B_R)} \big)^2 |x-y|^2. 
\end{multline}
Since
\begin{equation}\label{L2-part2}
\lim_{n \to + \infty}  \mathop{\iint_{B_R \times B_R}}_{\{|x - y| < 1\}} |x-y|^2 \rho_n(|x - y|) \, dx \, dy = 0, 
\end{equation}
it follows from \eqref{L2-part1} that
\begin{multline*}%\label{L1-part3}
\liminf_{n \to + \infty } \mathop{\iint_{B_R \times B_R}}_{\{|x-y| < 1 \}} \frac{|\Psi_u (x, y) -  \Psi_u(x, x)|^2}{|x - y|^2} \rho_n (|x - y|) \, dx \, dy \\[6pt]
\ge \liminf_{n \to + \infty } \mathop{\iint_{B_R \times B_R} }_{\{|x-y| < 1\}} \frac{\big| \big(\nabla u(y) - \i A(y) u(y) \big) \cdot (x-y)\big|^2}{|x-y|^2} \rho_n(|x-y|)\, dx \, dy. 
\end{multline*}
We have, by the definition of $Q_N$,  
\begin{multline}\label{L2-part3}
 \liminf_{n \to + \infty } \mathop{\iint_{B_R \times B_R} }_{ \{ |x-y| < 1\}} \frac{\big| \big(\nabla u(y) - \i A(y) u(y) \big) \cdot (x-y)\big|^2}{|x-y|^2} \rho_n(|x-y|)\, dx \, dy \\[6pt]
 \ge 2 Q_N \int_{B_{R-1}} |\nabla u(y) - \i A(y) u(y) |^2 \, dy. 
\end{multline}
By the arbitrariness of $R>1$ we get
\begin{equation*}%\label{L1-part1-1}
\liminf_{n \to + \infty } \mathop{\iint_{\R^{2N}}}_{\{|x-y| < 1 \}} \frac{|\Psi_u (x, y) -  \Psi_u(x, x)|^2}{|x - y|^2} \rho_n (|x - y|) \, dx \, dy \ge 2 Q_N \int_{\R^N} |\nabla u - \i A(x) u|^2 \, dx,  
\end{equation*}
which implies \eqref{stimasit}. 

Assertion \eqref{stimasit2} can be derived as follows. We have, by H\"older's inequality, 
\begin{multline*}
\mathop{\iint_{B_R \times B_R}}_{\{|x-y| < 1\}} \frac{|\Psi_u (x, y) -  \Psi_u(x, x)|^{2}}{|x - y|^{2}}  \rho_n(|x-y|) dx \, dy 
\\[6pt]
\le 
\left( \mathop{\iint_{B_R \times B_R}}_{\{|x-y| < 1 \}} \frac{|\Psi_u (x, y) -  \Psi_u(x, x)|^{2+\varepsilon_n}}{|x - y|^{2+\varepsilon_n}}  \rho_n(|x-y|) dx \, dy  \right)^{\frac{2}{2 + \varepsilon_n}} \left( \mathop{\iint_{B_R \times B_R}}_{\{|x-y| < 1\}}  \rho_n(|x-y|) \, dx \, dy  \right)^{\frac{\varepsilon_n}{ 2 + \varepsilon_n}}.  
\end{multline*}
Since, for every $R>0$, there holds
$$
\lim_{n \to + \infty}  \left( \mathop{\iint_{B_R \times B_R}}_{\{|x-y| \le 1\}}  \rho_n(|x-y|) \, dx \, dy  \right)^{\frac{\varepsilon_n}{ 2 + \varepsilon_n}} =1, 
$$
we get \eqref{stimasit2} from \eqref{L2-part3} and the arbitrariness of $R>1$.
\end{proof}

We are ready to prove \eqref{main-BBM-2}. 

\begin{lemma}[Limit formula] \label{lem-L3} Let $A: \R^N \to \R^N$ be  Lipschitz  
	and let $\{\rho_n\}_{n\in\N}$ be a sequence of nonnegative radial mollifiers.  Then, for $u \in H^1_A(\R^N)$,  
	\begin{equation*}
		\lim_{n \to +\infty } \iint_{\R^{2N}} \frac{|\Psi_u (x, y) -  \Psi_u(x, x)|^2}{|x - y|^2} \rho_n (|x - y|) \, dx \, dy = 2 Q_N \int_{\R^N} |\nabla u - \i A(x) u|^2 \, dx. 
	\end{equation*}
\end{lemma}

\begin{proof} By Lemma~\ref{lem-L1} and the density of $C^\infty_c(\R^N)$ in $H^1_A(\R^N)$, one might assume that $u \in C^2_c(\R^N)$. From Lemma~\ref{lem-L2}, it suffices to prove that, for  $u \in C^2_c(\R^N)$, 
\begin{equation}\label{L3-part0}
\limsup_{n \to +\infty } \iint_{\R^{2N}} \frac{|\Psi_u (x, y) -  \Psi_u(x, x)|^2}{|x - y|^2} \rho_n (|x - y|) \, dx \, dy \leq 2 Q_N \int_{\R^N} |\nabla u - \i A(x) u|^2 \, dx. 
\end{equation}
Fix $R>4$ such that $\supp u \subset B_{R/2}$.  Using \eqref{L2-part1} and \eqref{L2-part2}, 
one derives that
\begin{multline*}
 \limsup_{n \to + \infty } \mathop{\iint_{B_R \times B_R} }_{ \{ |x-y| < 1\}}\frac{|\Psi_u (x, y) -  \Psi_u(x, x)|^2}{|x - y|^2} \rho_n (|x - y|) \, dx \, dy \\[6pt]\le \limsup_{n \to + \infty} \mathop{\iint_{B_R \times B_R} }_{\{|x-y| < 1\}} \frac{\big| \big(\nabla u(y) - \i A(y) u(y) \big) \cdot (x-y)\big|^2}{|x-y|^2} \rho_n(|x-y|)\, dx \, dy, 
\end{multline*}
which yields 
\begin{multline}\label{L3-part2}
 \limsup_{n \to + \infty } \mathop{\iint_{B_R \times B_R} }_{ \{ |x-y| < 1\}}\frac{|\Psi_u (x, y) -  \Psi_u(x, x)|^2}{|x - y|^2} \rho_n (|x - y|) \, dx \, dy \\[6pt]  \le  2 Q_N \int_{\R^N} |\nabla u(y) - \i A(y) u(y) |^2 \, dy. 
\end{multline}
On the other hand, we have 
\begin{multline}\label{L3-part3}
\limsup_{n \to + \infty }\mathop{\iint_{\R^{2N}}}_{ \{ |x-y| \ge 1\}} \frac{|\Psi_u (x, y) -  \Psi_u(x, x)|^2}{|x - y|^2} \rho_n (|x - y|) \, dx \, dy   \\[6pt] \le  
\limsup_{n \to + \infty }\mathop{\iint_{\R^{2N}}}_{ \{ |x-y| \ge 1\}} 2 \big(|u(x)|^2 + |u(y)|^2\big)\rho_n(|x-y|)\, dx \, dy=0,
\end{multline}
and the fact that 
\begin{equation}\label{L3-part4}
\mbox{ if }  (x, y) \not \in B_R \times B_R    \mbox{ and }   |x - y | < 1 \mbox{ then }|\Psi_u (x, y) -  \Psi_u(x, x)| = 0, 
\end{equation}
by the choice of $R$. 
Combining \eqref{L3-part2}, \eqref{L3-part3}, and \eqref{L3-part4} yields \eqref{L3-part0}. 
 \end{proof}

The following result is about uniform bounds for the integrals in \eqref{main-BBM-1}. 

\begin{lemma}
	\label{lem-R1} 
	Let  $A: \R^N \to \R^N$ be Lipschitz  and let $\{\rho_n\}_{n\in\N}$ be a sequence of 
	nonnegative radial mollifiers. Then $u \in H^1_A(\R^N)$ if $u \in L^2(\R^N)$ and 
\begin{equation}\label{R1-assumption}
\sup_{n\in\N}  \iint_{\R^{2N}} \frac{|\Psi_u (x, y) -  \Psi_u(x, x)|^2}{|x - y|^2} \rho_n (|x - y|) \, dx \, dy < + \infty. 
\end{equation}
\end{lemma}

\begin{proof} Let $\{ \tau_m \}$ be a sequence of  nonnegative mollifiers  with $\supp \tau_m \subset B_1$ which is normalized by the condition $\int_{\R^N} \tau_m(x) \, dx  = 1$. Set 
$$
u_m = u * \tau_m. 
$$
We estimate 
\begin{equation*}
\iint_{\R^{2N}} \frac{|\Psi_{u_m} (x, y) -  \Psi_{u_m}(x, x)|^2}{|x - y|^2} \rho_n (|x - y|) \, dx \, dy.  
\end{equation*}
We have 
\begin{multline*}
\iint_{\R^{2N}} \frac{|e^{\i(x-y)\cdot A\left(\frac{x+y}{2}\right)} u_m(y) -  u_m(x)|^2}{|x - y|^2} \rho_n (|x - y|) \, dx \, dy  \\[6pt]
=  \iint_{\R^{2N}} \frac{\Big|\int_{\R^N} \Big(e^{\i(x-y)\cdot A\left(\frac{x+y}{2}\right)} u(y - z) -  u(x - z) \Big) \tau_m(z)  \, dz \Big|^2}{|x - y|^2} \rho_n (|x - y|) \, dx \, dy. 
\end{multline*}
By the change of variables $y' = y -z$ and $x'= x - z$ and using the inequality $|a+b|^2 \le 2 (|a|^2 + |b|^2)$ for all $a,b\in\C$ and applying Jensen's inequality, we deduce that   
\begin{align}\label{R1-part1}
\iint_{\R^{2N}} &  \frac{|\Psi_{u_m} (x, y) -   \Psi_{u_m}(x, x)|^2}{|x - y|^2}  \rho_n (|x - y|) \, dx \, dy \nonumber\\[6pt]
\le & 2 \iint_{\R^{2N}} \frac{|\Psi_{u} (x, y) -  \Psi_{u}(x, x)|^2}{|x - y|^2} \rho_n (|x - y|) \, dx \, dy \nonumber \\[6pt] + 
&  2 \int_{\R^N} \int_{\R^N} \int_{\R^N}  \frac{\Big|e^{\i(x-y)\cdot A\left(\frac{x+y}{2} + z\right)} - e^{\i(x-y)\cdot A\left(\frac{x+y}{2} \right)} \Big|^2  |u(y)|^2  }{|x - y|^2}  \tau_m(z)  \rho_n (|x - y|)\, dz  \, dx \, dy. 
\end{align}
Since, for $t \in \R$,  
$$
|e^{\i t} -1| \le C |t|,
$$
it follows that, for all $x, \, y, \, z \in\R^{N}$, 
\begin{multline*}
	%\label{assumption-A-1}
\Big|e^{\i(x-y)\cdot A\left(\frac{x+y}{2} + z\right)} - e^{\i(x-y)\cdot A\left(\frac{x+y}{2} \right)} \Big|  = \Big| e^{\i(x-y)\cdot \left( A\left(\frac{x+y}{2} + z\right) - A \left( \frac{x+y}{2}\right) \right)} -1 \Big| \\[6pt]
\le C  \|\nabla A\|_{L^\infty(\R^N)} |x -y| |z| \le C  |x -y| |z| . 
\end{multline*}
Here and in what follows in this proof,  $C$ denotes some positive constant independent of $m$ and $n$. 
Taking into account the fact that $\supp \tau_m \subset B_1$, we obtain 
\begin{multline}\label{R1-part2}
\int_{\R^N} \int_{\R^N} \int_{\R^N}  \frac{\Big|e^{\i(x-y)\cdot A\left(\frac{x+y}{2} + z\right)} - e^{\i(x-y)\cdot A\left(\frac{x+y}{2} \right)} \Big|^2  |u(y)|^2  }{|x - y|^2}  \tau_m(z)  \rho_n (|x - y|)\, dz  \, dx \, dy \\[6pt]
 \le  \int_{\R^N} \int_{\R^N} \int_{\R^N} C |u(y)|^2 \tau_m(z) \rho_n(|x-y|) \, dz \, dx \, dy \le C. 
\end{multline}
Combining \eqref{R1-assumption}, \eqref{R1-part1}, \eqref{R1-part2} yields 
\begin{equation}\label{R1-part3}
\iint_{\R^{2N}}\frac{|\Psi_{u_m} (x, y) -  \Psi_{u_m}(x, x)|^2}{|x - y|^2} \rho_n (|x - y|) \, dx \, dy  \le C. 
\end{equation}
On the other hand, by Lemma~\ref{lem-L2} we have
\begin{equation}\label{R1-part4}
\liminf_{n \to \infty }\iint_{\R^{2N}} \frac{|\Psi_{u_m} (x, y) -  \Psi_{u_m}(x, x)|^2}{|x - y|^2} \rho_n (|x - y|) \, dx \, dy  \ge 2 Q_N \int_{\R^N} |\nabla u_m - \i A(x) u_m|^2 \, dx. 
\end{equation}
The conclusion now immediately follows from \eqref{R1-part3} and \eqref{R1-part4} after letting $m \to + \infty$.
\end{proof}

We are ready to give the proof of Theorem~\ref{main-BBM}.

\medskip 
\noindent 
{\bf Proof of Theorem~\ref{main-BBM}.} Theorem~\ref{main-BBM} is a direct consequence of Lemmas \ref{lem-L1}, \ref{lem-L3} and \ref{lem-R1}. \qed

\begin{remark} \label{rem-Hs} \rm 
	Let $\{\rho_n\}_{n \in \N}$ be a sequence of  non-negative radial functions such that 
$$
\int_0^1 \rho_n(r) r^{N-1} \, d r = 1,  \quad \lim_{n \to + \infty} \int_\delta^1 \rho_n(r) r^{N-1} \, d r = 0,\quad \mbox{for every $\delta > 0$}, 
$$
and
$$
\lim_{n \to + \infty} \int_1^\infty \rho_n(r) r^{N-3} \, d r = 0. 
$$
Theorem~\ref{main-BBM} then holds for such a sequence $\{ \rho_n\}_{n\in\N}$ provided that  the constant $2$ in \eqref{main-BBM-3}
is replaced by an appropriate positive constant $C$ independent of $u$. This follows by taking into account the fact that, for $u \in L^2(\R^N)$,   
\begin{multline*}
\limsup_{n \to + \infty}\mathop{\iint_{\R^{2N}}}_{\{|x - y| \ge 1\}}   \frac{|\Psi_{u} (x, y) -   \Psi_{u}(x, x)|^2}{|x - y|^2}  \rho_n (|x - y|) \, dx \, dy  \\[6pt]
\le 2 \limsup_{n \to + \infty}\mathop{\iint_{\R^{2N}}}_{\{|x - y| \ge 1\}}  \big(|u(x)|^2 + |u(y)|^2 \big)  \rho_n (|x - y|) |x-y|^{-2} \, dx \, dy=0. 
\end{multline*}
For example, this applies to the radial sequence
$$
\rho_n(r) = 2 (1 - s_n) r^{2 -2s_n - N},\quad \mbox{for $r>0$}, 
$$
which provides a characterization of $H^1_A(\R^N)$ and yields
$$
\lim_{n \to + \infty}\, (1-s_n)\iint_{\R^{2N}}\frac{|u(x)-e^{\i (x-y)\cdot A\left(\frac{x+y}{2}\right)}u(y)|^2}{|x-y|^{N+2s_n}}dxdy
=2 Q_N  \int_{\R^N} |\nabla u - \i A(x) u|^2 \, dx.
$$
\end{remark}  

%\medskip 
%	Although the classical magnetic Sobolev spaces are typically used in the Hilbert case $H^1_A(\R^N)$, for the sake of completeness the above formula can be extended, with minor modifications, to the general case of the spaces $W^{1,p}_A(\R^N)$ (with obvious changes in the above definitions). 
	
\medskip 
\noindent
Consider now the space $(\mathbb{C}^n, |\cdot|_{p})$ ($n \ge 1$), endowed with the norm
\begin{equation*}
	|z|_p:=\big(|(\Re z_1,\ldots, \Re z_n)|^p+|(\Im z_1,\ldots, \Im z_n)|^p\big)^{1/p},
\end{equation*}
where $|\cdot|$ is the Euclidean norm of $\R^n$ and
$\Re a$, $\Im a$ denote the real and imaginary parts of $a\in\mathbb{C}$
respectively. 
We emphasize that this is not related to the $p$-norm in $\R^n$. In what follows, we use this notation with $n = N$ and $n=1$. 
Notice that $|z|_p=|z|$ whenever $z\in \R^n$, which makes our next 
statements consistent with the case $A=0$ and $u$ being 
a real valued function. Also $|\cdot|_2=|\cdot|,$ consistently with the previous definition. 
Define, for some ${\boldsymbol \omega} \in \mathbb{S}^{N-1}$, 	
\begin{equation}
		\label{Q-Np}
		Q_{N, p}:=\frac{1}{p}\int_{{\mathbb S}^{N-1}}| {\boldsymbol \omega}\cdot \sigma|_p^{p} \, d \sigma. 
\end{equation}
We have, for $z \in \mathbb{C}^N$, (see \cite{bourg,PSV}), 
\begin{equation}
\label{constant-z}
\int_{\mathbb{S}^{N-1}} |z \cdot \sigma|_p^p \, d \sigma = \int_{\mathbb{S}^{N-1}} |\Re z \cdot \sigma|^p \, d \sigma +  \int_{\mathbb{S}^{N-1}} |\Im z \cdot \sigma|^p \, d \sigma \\
= |\Re z|^p p Q_{N, p} +  |\Im z|^p p Q_{N, p}  =  |z|_p^p p Q_{N, p}. 
\end{equation}

\medskip
\noindent
Using the same approach and technique, one can prove the following $L^p$ version of Theorem~\ref{main-BBM}. 

\begin{theorem} \label{main-BBM-p} Let  $p\in (1,+\infty)$,  
$A:\R^N\to\R^N$ be Lipschitz, and let $\{\rho_n\}_{n\in\N}$ be a sequence of nonnegative radial mollifiers. 
Then $u \in W^{1, p}_A(\R^N)$ if and only if $u \in L^p(\R^N)$ and 
\begin{equation*}
\sup_{n \in \N}  \iint_{\R^{2N}} \frac{|\Psi_u (x, y) -  \Psi_u(x, x)|_p^p}{|x - y|^p} \rho_n (|x - y|) \, dx \, dy < + \infty. 
\end{equation*}
Moreover, for $u \in W^{1, p}_A(\R^N)$, we have
\begin{equation*}
\lim_{n \to + \infty} \iint_{\R^{2N}} \frac{|\Psi_u (x, y) -  \Psi_u(x, x)|_p^p}{|x - y|^p} \rho_n (|x - y|) \, dx \, dy = p Q_{N, p}  \int_{\R^N} |\nabla u - \i A(x) u|_p^p \, dx
\end{equation*}
and 
\begin{multline}\label{last-BBM-p}
\iint_{\R^{2N}} \frac{|\Psi_u (x, y) -  \Psi_u(x, x)|_p^p}{|x - y|^p} \rho_n (|x - y|) \, dx \, dy \\
\le C_{N, p} \int_{\R^N} |\nabla u - \i A(x) u|_p^p \, dx + C_{N, p} \big(2+ \|\nabla A \|_{L^\infty(\R^N)}^p \big)  \int_{\R^N} |u|_p^p \, dx, 
\end{multline}
for some positive constant $C_{N, p}$ depending only on $N$ and $p$. 
\end{theorem}

\begin{remark} \rm Assume that $C$ is a positive constant such that, for all $a, b \in \C$, 
$$
|a + b|_p^p \le C (|a|_p^p + |b|_p^p). 
$$
Then assertion \eqref{last-BBM-p} of Theorem~\ref{main-BBM-p} holds with $C_{N, p} =  |\mathbb{S}^{N-1}|C$.   
\end{remark}

\section{Proof of Theorem~\ref{main} and its $L^p$ version} \label{sect-H1A}

%\subsection{Preliminary results} \label{sect-Preliminary}

Let us set, for $\sigma \in \mathbb{S}^{N-1}$, 
$$
\M_{\sigma}(g, x):=\sup_{t>0}\frac{1}{t}\int_{0}^{t}\left|g(x + s \sigma)\right|ds. 
$$
and denote $\M_{e_N}$ by $\M_N$, $e_N:=(0,\ldots,0,1)$.
We have the following result which is a direct consequence of the theory of maximal functions, see e.g., \cite[Theorem 1, page 5]{Stein}.

\begin{lemma}[Maximal function estimate]
	\label{maxprop}
There exists a universal constant  $C>0$ such that, for all $\sigma \in \mathbb{S}^{N-1}$, 
$$
\int_{\R^N} |\M_{\sigma}(g, x)|^2dx\leq C \int_{\R^N} |g|^2dx, \quad\text{for all $g\in L^2(\R^N)$}.
$$	
\end{lemma}

The following lemma yields an upper bound of $J_\delta(u)$ in terms of the norm of $u$ in $H^1_A(\R^N)$. 

\begin{lemma}[Uniform upper bound]
	\label{lem-variant}
	Let   $A:\R^N\to\R^N$ be Lipschitz and  $u\in H^1_A(\R^N)$.  We have 
\begin{equation*}
%\label{UB}
\sup_{\delta > 0} J_\delta (u)
\leq C_N  \left( \int_{\R^N}|\nabla u-\i A(x)u|^2\, dx + \big(\|\nabla A\|_{L^\infty(\R^N)}^2 + 1\big) \int_{\R^N} |u|^2  \, dx \right). 
\end{equation*}
\end{lemma}

\begin{proof} 
%Without loss of generality,  we may prove the statement for functions  $u\in C^\infty_c(\R^N)$. In fact, given 
%$u\in H^1_A(\R^N)$, by the density of  $C^\infty_c(\R^N)$ in $H^1_A(\R^N)$ (cf.\ \cite[Proposition 2]{EstLio}), 
%if $\{\varphi_h\}_{h\in\N}\subset C^\infty_c(\R^N)$ and $\varphi_h\to u$ in $H^1_A(\R^N)$ as $h\to\infty$, then 
%up to a subsequence $\varphi_h(x)\to u(x)$ a.e., so for a.e.\ $x\in\R^N$
%\begin{align*}
%	|\Psi_{\varphi_h}(x,y)-\Psi_{\varphi_h}(x,x)|&=|\varphi_h(x)-e^{\i (x-y)\cdot A\left(\frac{x+y}{2}\right)}\varphi_h(y)| \\
%	&\to |u(x)-e^{\i (x-y)\cdot A\left(\frac{x+y}{2}\right)}u(y)|=|\Psi_{u}(x,y)-\Psi_{u}(x,x)|.
%\end{align*}
%Hence, for a.e.\ $(x,y)\in\R^{2N}$ we obtain, for any fixed $\delta>0$,
%$$
%\frac{\delta^2}{|x-y|^{N+2}}{\bf 1}_{\{|\Psi_{\varphi_h}(x,y)-\Psi_{\varphi_h}(x,x)|>\delta\}}(x,y)\to \frac{\delta^2}{|x-y|^{N+2}}{\bf 1}_{\{|\Psi_{u}(x,y)-\Psi_{u}(x,x)|\geq \delta\}}(x,y).
%$$
%Therefore, if the statement holds for functions in $C^\infty_c(\R^N)$, we have, for each $\delta >0$, 
%\begin{multline*}
% \iint_{\R^{2N}} \frac{\delta^2}{|x-y|^{N+2}}{\bf 1}_{\{|\Psi_{\varphi_h}(x,y)-\Psi_{\varphi_h}(x,x)|>\delta\}}(x,y)dx \, dy \\
% \leq  C_N  \Big( \int_{\R^N}|\nabla \varphi_h-\i A(x)\varphi_h|^2\, dx + (\| \nabla A\|_{L^\infty(\R^N)}^2 +1)\int_{\R^N} |\varphi_h|^2 \ dx\Big),
%\end{multline*}
%and we then let $h \to \infty$ to obtain the corresponding inequality for $u$, via strong convergence on the right-hand side and Fatou's lemma on the left-hand side. 

By density of $C^\infty_c(\R^N)$ in $H^1_A(\R^N)$, using Fatou's lemma, 
we can assume that $u \in C^1_c(\R^N)$. 
For each $\delta>0$, let us define
$$
%\varphi(y) := e^{\i (x-y)\cdot A\left(\tfrac{x+y}{2}\right)}u(y),\qquad
\A_{\delta}:= \Big\{ (x,y)\in \R^{2N}: |\Psi_u (x,y)-\Psi_u(x,x)| > \delta, \; |x -y| < 1 \Big\}
$$
and
$$
\mathcal{B}_{\delta}:= \Big\{(x,y)\in \R^{2N}: |\Psi_u(x,y)-\Psi_u(x,x)|>\delta, \;  |x - y| \ge 1 \Big\}.$$
We have
$$
\iint_{\R^{2N}} \frac{\delta^2}{|x-y|^{N+2}}{\bf 1}_{\mathcal{B}_{\delta}} \, dx \, dy  \le \iint_{\R^{2N}}  \frac{|\Psi_u(x, y) - \Psi_u(x,x)|^2}{|x-y|^{N+2}} {\bf 1}_{\{|x - y| \ge 1\}} \, dx \, dy. 
$$
Since $|\Psi_u(x, y) - \Psi_u(x, x)| \le |u(x)| + |u(y)|$ and 
$$
\mathop{\iint_{\R^{2N}}}_{\{|x - y| \ge 1 \}}
\frac{|u(x)|^2}{|x-y|^{N+2}}dx \, dy\leq  C_N  \int_{\R^N} |u(x)|^2 \, dx,
$$ 
it follows that 
$$
\iint_{\R^{2N}} \frac{\delta^2}{|x-y|^{N+2}}{\bf 1}_{\mathcal{B}_{\delta}} \, dx \, dy  \le  C_N \int_{\R^N} |u(x)|^2 \, dx. 
$$
We are therefore interested in estimating the integral
$$
\iint_{\A_{\delta}} \frac{\delta^2}{|x-y|^{N+2}} dx \, dy.
$$
Let us now define
$$
\mathcal{X}_\delta:= \Big\{ (x,h,\sigma)\in \R^N \times (0,1)\times \mathbb{S}^{N-1}: |\Psi_u (x,x+h\sigma) - \Psi_u(x,x)|>\delta \Big\}.
$$
Performing the change of variables  
$y= x + h \sigma$, for $h \in (0,1)$ and $\sigma \in \mathbb{S}^{N-1}$,
yields
\begin{equation*}
\iint_{\A_{\delta}} \frac{\delta^2}{|x-y|^{N+2}} dx \, dy = 
\iiint_{\mathcal{X}_\delta} \frac{\delta^2}{h^{3}} \, dh  \, dx  \, d\sigma=
\int_{\mathbb{S}^{N-1}}\iint_{\CC_{\sigma}} \frac{\delta^2}{h^{3}} \, dh  \, dx  \, d\sigma,
\end{equation*}
where $\CC_{\sigma}$ denotes the set
$$
\CC_{\sigma}:= \Big\{ (x,h) \in \R^N \times (0, 1): |\Psi_u (x,x+h\sigma)-\Psi_u (x,x)|>\delta \Big\},
\quad \sigma \in \mathbb{S}^{N-1}. 
$$
Without loss of generality 
it suffices  to prove that, for $\sigma = e_N = (0,\ldots,0,1) \in \mathbb{S}^{N-1}$, 
\begin{equation}\label{directionN-1}
\iint_{\CC_{e_N}}\frac{\delta^2}{h^{3}}dh dx \leq C_N \left( \int_{\R^N}|\nabla u  - \i A(x) u|^2 dx 
+  \|\nabla A \|_{L^\infty(\R^N)}^2  \int_{\R^N} |u|^2\, dx \right).
\end{equation}
We have, by virtue of \eqref{p1-0}, 
%\footnote{From H.-M.:Please check this point again; the constant $C_N$ is taken out in %comparison with the previous version!. Corresponding modifications are made for %\eqref{part-1}, \eqref{part-2}, the definition of $F (x')$ in the proof of Lemma 3.3.}
\begin{equation}
\label{key_p}
\Big|\Psi(x, x + h e_N) - \Psi(x, x) \Big| \le h \M_{N}( |\nabla u - \i A u|, x) +h^2 \|\nabla A \|_{L^\infty(\R^N)} \M_{N}(|u|, x). 
\end{equation}
Using the fact that  if $a + b > \delta$ then either $a > \delta/2$ or $b > \delta/2$, we derive that 
\begin{align*}
\iint_{\CC_{e_N}}\frac{\delta^2}{h^{3}}dh dx &\le \iint_{ \{h \M_{N}( |\nabla u - \i A u|, x) >  \delta / 2\}} \frac{\delta^2}{h^{3}}\, dh \, dx  
+ \iint_{ \{ h^2 \| \nabla A\|_{L^\infty(\R^N)} \M_N (|u|, x ) > \delta/ 2\}} \frac{\delta^2}{h^{3}} \,dh \,  dx\\
&\le \iint_{\{ h \M_{N}( |\nabla u - \i A u|, x) >  \delta / 2\}} \frac{\delta^2}{h^{3}}\, dh \, dx  
+ \iint_{ \{h \| \nabla A\|_{L^\infty(\R^N)} \M_N (|u|, x ) > \delta/ 2\}} \frac{\delta^2}{h^{3}} \,dh  \, dx, 
\end{align*}
where the last inequality follows recalling that since 
$(x, h) \in \CC_{e_N}$ then $h \in (0, 1)$.
As usual, by using the theory of maximal functions stated in Lemma~\ref{maxprop}, we have 
\begin{equation}\label{part-1}
\iint_{\{ h \M_{N}( |\nabla u - \i A u|, x) >  \delta/2\}} \frac{\delta^2}{h^{3}}dh dx 
\le C_N \int_{\R^N} |\nabla u - \i A (x) u|^2 \, dx
\end{equation}
and 
\begin{equation}\label{part-2}
 \iint_{ \{h  \| \nabla A\|_{L^\infty(\R^N)}  \M_N (|u|, x )> \delta/2 \}} \frac{\delta^2}{h^{3}} \,dh dx \le C_N \| \nabla A\|_{L^\infty}^2 \int_{\R^N}|u|^2 \, dx. 
\end{equation}
 Assertion \eqref{directionN-1} follows from \eqref{part-1} and \eqref{part-2}. The proof is complete.
\end{proof}

We next establish 

\begin{lemma}[Limit formula] 
	\label{formul}  
	Let   $A:\R^N\to\R^N$ be Lipschitz	
and 	$u \in H^1_A(\R^N)$.	Then 
	\begin{equation*}
	\lim_{\delta\searrow 0} J_\delta (u)= 	Q_{N}\int_{\R^N}|\nabla u-\i A(x)u|^2\, dx,
	\end{equation*}
	where $Q_N$ is the constant defined in \eqref{valoreK}.
\end{lemma}
\begin{proof}  
	By virtue of Lemma~\ref{lem-variant}, for every $\delta>0$ and all $w \in H^1_A(\R^N)$,
	we have
\begin{equation}\label{Limit formula-1}
J_{\delta}(w)
\leq C_N  \left( \int_{\R^N}|\nabla w -\i A(x) w |^2\, dx + \big(\|\nabla A\|_{L^\infty(\R^N)}^2 + 1 \big) \int_{\R^N} |w|^2  \, dx \right).
\end{equation}
Since 
$$
|\Psi_u(x, y) - \Psi_u(x, x)| \le |\Psi_v(x, y) - \Psi_v(x, x)| + |\Psi_{u-v}(x, y) - \Psi_{u-v}(x, x)|, 
$$
it follows that, 
for every $\varepsilon\in (0,1)$, 
\begin{multline*}
J_\delta(u) \leq \iint_{\left\{|\Psi_{v}(x,y)-\Psi_{v}(x,x)|>(1-\varepsilon)\delta\right\}}\frac{\delta^2}{|x-y|^{N+2}} dx \, dy \\[6pt]
+\iint_{\left\{|\Psi_{u-v}(x,y)-\Psi_{u-v}(x,x)|>\varepsilon\delta\right\}}\frac{\delta^2}{|x-y|^{N+2}} dx \, dy. 
\end{multline*}
This implies, for $\varepsilon \in (0, 1) $ and $u, \, v \in H^1_A(\R^N)$,   
\begin{equation}\label{Limit formula-2}
J_\delta(u) \le  (1-\varepsilon)^{-2}J_{(1-\varepsilon)\delta}(v)+ \varepsilon^{-2}J_{\varepsilon\delta}(u-v).
\end{equation}
%$$
%J_\delta(u) \le (1-\eps)^{-2} J_{(1 - \eps) \delta}(v) + \eps^{-2} J_{\eps \delta}(v - u) \mbox{ for } \eps \in (0, 1), 
%$$
From \eqref{Limit formula-1} and \eqref{Limit formula-2}, we derive that,  for $u, u_n \in H^1_A(\R^N)$ and $\eps \in (0, 1)$, 
\begin{multline}\label{lem-formula-part1}
J_{\delta}(u) - (1-\eps)^{-2} J_{(1-\eps) \delta} (u_n) \\[6pt]
\le \eps^{-2} C_N  \left( \int_{\R^N}|\nabla (u - u_n) -\i A(x) (u -u_n) |^2\, dx + \big(\|\nabla A\|_{L^\infty(\R^N)}^2 + 1\big) \int_{\R^N} |u-u_n|^2  \, dx \right)
\end{multline}
and 
\begin{multline}\label{lem-formula-part2}
(1 - \eps)^2 J_{\delta/ (1 - \eps)}(u_n) - J_{ \delta} (u) \\[6pt] \le \eps^{-2} C_N  \left( \int_{\R^N}|\nabla (u - u_n) -\i A(x) (u -u_n) |^2\, dx + \big(\|\nabla A\|_{L^\infty(\R^N)}^2 + 1\big) \int_{\R^N} |u-u_n|^2  \, dx \right). 
\end{multline}
Since $C^1_{c}(\R^N)$ is dense in $H^1_A(\R^N)$, from \eqref{lem-formula-part1} and \eqref{lem-formula-part2}, it suffices to prove 
the assertion for $u \in C^1_{c}(\R^N)$. This fact is assumed from now on.  

%		We recall that 
%	$$
%	\dfrac{\partial \Psi_u}{\partial y_N}(x,x) = \dfrac{\partial u}{\partial y_N}(x) -\i A_N(x)u(x),
%	$$
%	and that $u(x',\cdot)\in H^1_{{\rm loc}}(\R)$ for almost all $x'\in\R^{N-1}$, since
%	$u$ belongs to $H^1_{{\rm loc}}(\R^N)$ for every function $u \in H^1_A(\R^N)$ and
%	for all $(x_N,h)\in\R\times (0,\infty)$
%	$$
%	u(x+he_{N})-u(x)=\int_{x_N}^{x_N+h}\frac{\partial u}{\partial x_N}(x',s)ds.
%	$$
%	for almost any $x'\in\R^{N-1}$.
	Let $R>0$ be such that $\supp u \subset B_{R/2}$. We claim that, for every $\sigma \in \mathbb{S}^{N-1}$, there holds
	\begin{equation}
	\label{Claim0}
	\lim_{\delta \searrow 0} \iint_{\left\{(x,h)\in B_R \times(0,\infty): \Big|\tfrac{\Psi_u(x,x+\delta h \sigma) -\Psi_u(x,x)}{\delta h} \Big|h >1\right\}}\dfrac{1}{h^3} dh dx 
	= \dfrac{1}{2} \, \int_{\R^N} |(\nabla u - \i A u)\cdot \sigma|^2 dx.
	\end{equation}
	Without loss of generality, we can assume $\sigma = e_N \in \mathbb{S}^{N-1}$.
	Then, we aim to prove that
	\begin{equation*}
	%\label{Claim
	\lim_{\delta \searrow 0} \iint_{\left\{(x,h)\in B_R\times(0,\infty): \left|\tfrac{\Psi_u(x,x+\delta h e_N) -\Psi_u(x,x)}{\delta h} \right|h >1\right\}}\dfrac{1}{h^3} dh dx= \dfrac{1}{2} \, \int_{\R^N} \left|\dfrac{\partial u}{\partial y_N}(x) -\i A_N(x)u(x)\right|^2 dx,
	\end{equation*}
	where $A_N$ denotes the $N$-th component of $A$. To this end, we consider the sets
	\begin{align*}
	& \mathcal{C}_{e_N}(x',\delta) := \Big\{ (x_N,h) \in \R \times (0,\infty):  \left| \dfrac{\Psi_u(x,x+\delta h e_N) - \Psi_u(x,x)}{\delta h} \right| h >1 \Big\},   \\
	& \mathcal{E}(x') := \Big\{ (x_N,h) \in \R \times (0,\infty):  \left| \dfrac{\partial \Psi_u}{\partial y_N}(x,x)\right| h >1 \Big\}, \\
	& \mathcal{F}(x') := \Big\{(x_N,h) \in \R \times (0,\infty):  h \M_{N}(|\nabla u - \i A u|,x) + h^2 \|\nabla A\|_{L^\infty(\R^N)} \M_{N}(|u|,x) >1\Big\}. 
	\end{align*}
	Therefore, we obtain $\chi_{\mathcal{C}_{e_N}(x',\delta)} (x_N,h) \leq \chi_{\mathcal{F}(x')}(x_N,h)$ for a.e.\ $(x,h) \in B_R \times (0,\infty)$ (by \eqref{key_p} in the proof of Lemma~\ref{lem-variant}) and
	\begin{equation*}
	\int_{B_R}\int_{0}^{\infty} \dfrac{1}{h^3} \chi_{\mathcal{F}(x')}(x_N,h) \, dh \, dx \leq 
	\mathcal{I}_1 + \mathcal{I}_2,
	\end{equation*}
	where we have set
	\begin{align*}
	\mathcal{I}_1&:=\iint_{\big\{(x,h)\in B_R \times(0,\infty):\,  \M_{N}(|\nabla u - \i A u|,x) h > 1/2\big\}}\dfrac{1}{h^3} \, dh \, dx, \\
	\mathcal{I}_2&:=\iint_{\big\{(x,h)\in B_R \times(0,\infty):\,  h^2 \|\nabla A\|_{L^\infty(\R^N)}\M_{N}(|u|, x)> 1/2 \big\}}\dfrac{1}{h^3}\, dh \,  dx, 
	\end{align*}
and we have denoted $\chi$ the characteristic function.   We have, by the theory of maximal functions, 
	\begin{equation*}
	\mathcal{I}_1 \leq C \, \int_{\R^N}|\nabla u -\i A(x) u|^2 dx,
	\end{equation*}
	and,  by a straightforward computation, 
\begin{equation*}
	\mathcal{I}_2 \leq C \| \nabla A \|_{L^\infty(\R^N)} \| u\|_{L^\infty(\R^N)} |B_R|. 
	\end{equation*}
	The validity of Claim \eqref{Claim0} with $\sigma = e_N$ now follows from Dominated Convergence theorem since
	$$
	\lim_{\delta \searrow 0}\chi_{\mathcal{C}_{e_N}(x',\delta)} (x_N,h) = \chi_{\mathcal{E}(x')}(x_N,h), \quad \textrm{for a.e. } (x,h) \in B_R \times (0,\infty), 
	$$
and, by a direct computation, 
	$$
	\int_{B_R}\int_0^\infty
	\chi_{\mathcal{E}(x')} (x_N,h)
	\dfrac{1}{h^3} dh  dx=
	\dfrac{1}{2} \int_{B_R} \left|\dfrac{\partial u}{\partial y_N}(x) -\i A_N(x)u(x)\right|^2 dx.
	$$
	Now, performing a change of variables we get
	$$
	\iint_{\{|\Psi_u(x,y)-\Psi_u(x,x)|>\delta,\, \, x \in B_R\}}\frac{\delta^2}{|x-y|^{N+2}} \, dx \, dy = 
	\int_{B_R}\int_{\mathbb{S}^{N-1}}\int_0^\infty
	\chi_{\mathcal{C}_{\sigma}(\delta)} (x,h)
	\dfrac{1}{h^3}\,  dh \, d\sigma \, dx,
	$$
where 
$$
\mathcal{C}_{\sigma}(\delta) := \Big\{ (x,h) \in B_R \times (0,\infty):  \left| \dfrac{\Psi_u(x,x+\delta h \sigma) - \Psi_u(x,x)}{\delta h} \right| h >1 \Big\}.
$$
	Exploiting \eqref{Claim0}, we obtain 
	\begin{equation}\label{Limit-part1}
	\lim_{\delta \searrow 0} \iint_{\{|\Psi_u(x,y)-\Psi_u(x,x)|>\delta, \, x \in B_R\}}\frac{\delta^2}{|x-y|^{N+2}} \, dx \, dy = \dfrac{1}{2} \int_{\mathbb{S}^{N-1}}\int_{B_R} |(\nabla u -\i A u)\cdot \sigma|^2  \, dx \,  d\sigma.
	\end{equation}
On the other hand, since $\supp u \subset B_{R/2}$,  we have 
\begin{multline}\label{Limit-part2}
\lim_{\delta \searrow 0} \iint_{\{|\Psi_u(x,y)-\Psi_u(x,x)|>\delta, \, x \in \R^N \setminus B_R\}}\frac{\delta^2}{|x-y|^{N+2}} \, dx \, dy \\
= \lim_{\delta \searrow 0} \iint_{\{x \in \R^N \setminus B_R, \; y \in B_{R/2}\}}\frac{\delta^2}{|x-y|^{N+2}} \, dx \, dy =0.
\end{multline}
Combining \eqref{Limit-part1} and \eqref{Limit-part2} yields 
\begin{equation*}
	\lim_{\delta \searrow 0} \iint_{\{|\Psi_u(x,y)-\Psi_u(x,x)|>\delta\}}\frac{\delta^2}{|x-y|^{N+2}}\, dx \, dy = \dfrac{1}{2} \int_{\mathbb{S}^{N-1}}\int_{\R^N} |(\nabla u -\i A u)\cdot \sigma|^2 \, dx \, d\sigma.
	\end{equation*}
	In order to conclude, we  notice the following, see \eqref{constant-z}, 
	$$
	\int_{\mathbb{S}^{N-1}}|V \cdot \sigma|^2  \, d\sigma= 2 Q_N |V|^2, \quad \textrm{for any } V \in \C^N,
	$$
	where $Q_N$ is the constant defined in \eqref{valoreK}.
\end{proof}

We next deal with \eqref{bound}. 

\begin{lemma} 
	\label{lem-R2} Let $u \in L^2(\R^N)$ and let  $A: \R^N \to \R^N$ be Lipschitz. Then  $u \in H^1_A(\R^N)$ if
\begin{equation}
	\label{R2-assumption}
\sup_{\delta \in (0, 1)} J_\delta(u)  < + \infty. 
\end{equation}
\end{lemma}

\begin{proof} The proof is divided into two steps. 

\noindent
{\bf Step 1.} We  assume that $u \in L^2(\R^N)\cap L^\infty(\R^N)$. Set 
\begin{equation*}
	%\label{def-Idelta}
L := \sup_{x, y \in \R^N}|\Psi_u(x,y)-\Psi_u(x,x)|.
\end{equation*}
In light of \eqref{R2-assumption}, we obtain 
$$
\int_0^L \eps \delta^{\eps  - 1} J_\delta (u) \, d \delta \le C, 
$$
for some positive constant $C$ independent of $\eps\in (0,1)$.  
By Fubini's theorem and by the definition of $L$, we have
$$
\int_0^L \eps \delta^{\eps  - 1} J_\delta (u) \, d \delta  =  \int_{\R^{2N}} \frac{1}{|x-y|^{N+2}}   \int_0^{|\Psi_u(x,y) - \Psi_u(x, x)|} \eps \delta^{\eps + 1} \, d \delta\, dx \, dy. 
$$
It follows that
$$
\frac{1}{2 + \eps}\iint_{\R^{2N}} \frac{|\Psi_{u} (x, y) -  \Psi_{u}(x, x)|^{2 + \eps}}{|x - y|^{2+\eps}}  \frac{\eps}{|x-y|^{N  -\eps}} \, dx \, dy \le C. 
$$
By virtue of inequality \eqref{stimasit2} of Lemma~\ref{lem-R1}, we have 
\begin{equation*}
\liminf_{\eps \to 0} \iint_{\R^{2N}} \frac{|\Psi_{u} (x, y) -  \Psi_{u}(x, x)|^{2 + \eps}}{|x - y|^{2+\eps}}  \frac{\eps}{|x-y|^{N -\eps}} \, dx \, dy \ge 2 Q_N \int_{\R^N} |\nabla u - \i A(x) u |^2 \, dx,
\end{equation*}
which implies $u  \in H^1_A(\R^N)$. 

\medskip
\noindent
{\bf Step 2.} We consider the general case. For $M>1$, define ${\mathcal T}_M:\C \to \C$ by setting
\begin{equation*}
	%\label{def-TM}
{\mathcal T}_M(z) := \left\{\begin{array}{cl} z & \mbox{if $|z| \leq M$}, \\[4pt]
M z/ |z| & \mbox{otherwise}, 
\end{array}\right.
\end{equation*}
and denote
\begin{equation*}
	%\label{def-uM}
u_M := {\mathcal T}_M(u). 
\end{equation*}
Then, we have 
%\footnote{From H.-M.: Please check this point again! The constant $\gamma$ is taken out in %comparison with the previous version.}
\begin{equation*}
	%\label{R2-claim}
|{\mathcal T}_M(z_1) - {\mathcal T}_M(z_2)| \le  |z_1 - z_2|,\quad 
\mbox{ for all } z_1, z_2 \in \C. 
\end{equation*} 
It follows that 
$$
|\Psi_{u_M} (x, y) - \Psi_{u_M}(x, x)| \le |\Psi_u (x, y) - \Psi_u(x, x)|,\quad 
\mbox{ for all } x, \, y \in \R^N. 
$$
Hence we  obtain 
\begin{equation}\label{R2-part1}
  J_\delta(u_M)
 \le  J_{\delta}(u). 
\end{equation}
Applying the result in Step 1, we have $u_M \in H^1_A(\R^N)$ and hence by Lemma~\ref{formul}, 
\begin{equation}\label{R2-part2}
\lim_{\delta \to 0} J_\delta(u_M) = 2 Q_N \int_{\R^N} |\nabla u_M(x) - \i A(x) u_M(x)|^2 \, dx. 
\end{equation}
Combining \eqref{R2-part1} and \eqref{R2-part2} and letting $M \to + \infty$, we derive that $u \in H^1_A(\R^N)$. The proof is complete. 
\end{proof}

\begin{remark} \rm Similar approach used for $H^1(\R^N)$ is given  in \cite{nguyen06}. 
\end{remark}

\noindent{\bf Proof of Theorem~\ref{main}.} The limit formula stated in Theorem~\ref{main} follows by Lemma~\ref{formul}.
Now, if $u\in H^1_A(\R^N)$, then \eqref{main-estimate} follows from Lemma~\ref{lem-variant}. 
On the contrary, if $u\in L^2(\R^N)$ 
and \eqref{bound} holds, it follows from 
Lemma~\ref{lem-R2} that $u\in H^1_A(\R^N)$.
\qed

\medskip 
\noindent
Given $u$ a measurable complex-valued function,  define, for $1 <  p < + \infty$,  
$$
J_{\delta, p} (u): = \iint_{\{|\Psi_u(x,y)-\Psi_u(x,x)|_p>\delta\}}\frac{\delta^p}{|x-y|^{N+p}}\, dxdy, \quad \mbox{ for } \delta > 0.  
$$

	We have the following $L^p$-version of Theorem~\ref{main}. 

	\begin{theorem}
	\label{main2} Let $ p \in (1,  + \infty)$ and let $A:\R^N\to\R^N$ be Lipschitz. Then $u\in W^{1,p}_A(\R^N)$ if and only if $u \in L^p(\R^N)$ and 
	\begin{equation*}
	\sup_{0<\delta<1} J_{\delta, p} (u) <\infty. 
	\end{equation*}
	Moreover, we have, for $u \in W^{1, p}_A(\R^N)$, 
	\begin{equation*}
	\lim_{\delta\searrow 0} J_{\delta, p} (u) = 	Q_{N,p}\int_{\R^N}|\nabla u-\i A(x)u|_p^p\, dx
	\end{equation*}
	and
	\begin{equation*}
		J_{\delta, p} (u)
\leq C_{N, p}  \left( \int_{\R^N}|\nabla u-\i A(x)u|_p^p\, dx + \big(\|\nabla A\|_{L^\infty(\R^N)}^p + 1 \big) \int_{\R^N} |u|_p^p  \, dx \right),  
	\end{equation*}
	for some positive constant $C_{N, p}$ depending only on $N$ and $p$. 
	\end{theorem}

Recall that $Q_{N, p}$ is defined by \eqref{Q-Np}. 

\begin{proof} 
%The norm $|\cdot|_p$ defined on $\C^N$ is equivalent to the standard one $|\cdot|$, i.e.\ there
%	is $C=C(N,p)>0$ with $C^{-1} |z| \leq |z|_p \leq C |z|$, for all $z \in \C^N$.
%	Furthermore, we also have $|z\cdot w|_p\leq C|z|_p|w|_p$, for all $z,w \in \C^N$, where $z\cdot w=\sum_{i=1}^N z_iw_i$.
%	Finally, this norm satisfies the following crucial property for complex vectors (see the proof of \cite[Lemma 4.4]{PSV})
%\begin{equation}
%\label{p-form-cruc}
%\int_{\mathbb{S}^{N-1}}|V \cdot \sigma|^p_p \,d\sigma= p Q_{N,p} |V|^p_p, \quad \textrm{for any } V \in \C^N,
%\end{equation}
%where $Q_{N,p}$ is the constant defined in \eqref{Qnp}.	
We have the maximal function estimates in the form
	$$
	\int_{\R^N} |\M_{\sigma}(g, x)|_p^p dx\leq C_p \int_{\R^N} |g|_p^p dx, \quad\,\,\text{for all $g\in L^p(\R^N)$}.
	$$
	for all $\sigma \in \mathbb{S}^{N-1}$ and $g \in L^{p}(\R^N)$, either complex or real valued.
It is readily checked (repeat the proof of \cite[Theorem 7.22]{LL} with straightforward adaptations) 
that $C^{\infty}_{c}(\R^N)$ is dense in $W^{1,p}_{A}(\R^N)$. Lemma~\ref{lem-variant} holds in the modified form
\begin{equation*}
%\label{p-UB}
%\sup_{\delta>0}
J_{\delta, p} (u)
\leq C_{N, p}  \left( \int_{\R^N}|\nabla u-\i A(x)u|_p^p\, dx + \big(\|\nabla A\|_{L^\infty(\R^N)}^p + 1 \big) \int_{\R^N} |u|_p^p  \, dx \right), 
\end{equation*}
for all $u\in W^{1,p}_A(\R^N)$ and $\delta>0$. To achieve
this conclusion, it is sufficient to observe that, see \eqref{key_p}, 
\begin{equation*}
%\label{p-key_p}
\big|\Psi(x, x + h e_N) - \Psi(x, x) \big|_p \le  h \M_{N}( |\nabla u - \i A u|_p, x) +  h^2 \|\nabla A \|_{L^\infty(\R^N)} \M_{N}(|u|_p, x). 
\end{equation*}
The rest of the proof follows verbatim. Lemma~\ref{formul} holds in the form 
\begin{equation*}
\lim_{\delta\searrow 0} J_{\delta,p} (u)= 	Q_{N,p}\int_{\R^N}|\nabla u-\i A(x)u|_p^p\, dx,
\end{equation*}
for every $u \in W^{1,p}_A(\R^N)$. In fact, mimicking the proof of Lemma~\ref{formul},   one obtains
\begin{equation*}
\lim_{\delta \searrow 0} \iint_{\{|\Psi_u(x,y)-\Psi_u(x,x)|_p>\delta\}}\frac{\delta^p}{|x-y|^{N+p}} \, dx \, dy = \dfrac{1}{p} \int_{\mathbb{S}^{N-1}}\int_{\R^N} |(\nabla u -\i A u)\cdot \sigma|^p \, dx  \, d\sigma.
\end{equation*}
The final conclusion follows from \eqref{constant-z}. 
Lemma~\ref{lem-R2} can be modified 
accordingly with minor modifications, replacing $|\cdot|$ with $|\cdot|_{p}$. 
%\footnote{H.-M. makes some small modifications here.} 
\end{proof}

\section{Convergence almost everywhere and convergence in $L^1$} \label{sect-pointwise}

\medskip 

Motivated by the work in \cite{BHN2} (see also \cite{PS}), we are interested in  other modes of convergence in the context of Theorems~\ref{main-BBM} and \ref{main}. 
We only consider the case $p=2$. Similar results hold for $ p \in (1,  + \infty)$ with similar proofs.  We begin with the corresponding results related to Theorem~\ref{main-BBM}.  For $u \in L^1_{\loc}(\R^N)$, set  
$$
D_n(u, x): = \int_{\R^{N}} \frac{|\Psi_u (x, y) -  \Psi_u(x, x)|^2}{|x - y|^2} \rho_n (|x - y|) \, dy,\quad \mbox{ for } x \in \R^N.  
$$
We have 
\begin{proposition}\label{pro-BBM} Let $A: \R^N \to \R^N$ be Lipschitz, $u \in H^1_A(\R^N)$, and let $(\rho_n)$ be a sequence of radial mollifiers such that 
\begin{equation*}
\sup_{t > 1} \sup_n  t^{-2} \rho_n(t) < + \infty.  
\end{equation*}
We have 
\begin{equation*}
\lim_{n \to + \infty} D_n(u, x) = 2 Q_{N} |\nabla u(x) - \i A(x) u(x)|^2,\quad 
\mbox{for a.e.\ $x \in \R^N$},
\end{equation*}
and 
\begin{equation*}
\lim_{n \to + \infty} D_n(u, \cdot) =  2 Q_{N} |\nabla u( \cdot ) - \i A(\cdot ) u(\cdot )|^2, 
\quad\mbox{in $L^1(\R^N)$}.  
\end{equation*}
\end{proposition}

Before giving the proof of Proposition~\ref{pro-BBM}, we  recall the following result established in \cite[Lemma 1]{BHN3} (see also\cite[Lemma 2]{BHN2} for a more general version). 

\begin{lemma} \label{lem-maximal-BN} Let $r>0$,  $x \in \R^N$ 
	and $f \in L^1_{{\rm loc}}(\R^N)$.  We have
\begin{equation*}
%\label{maximal-part1}
\int_{\mathbb{S}^{N-1}} \int_0^r |f(x + s \sigma)| \, d s \, d \sigma \le C_N r M(f)(x). 
\end{equation*}
\end{lemma}

\noindent
Here and in what follows,  for $x \in \R^N$ and $r >0$, let $B_x(r)$ denote the open ball in $\R^N$ centered at $x$ and of radius $r$. Moreover, 
$M(f)$ denotes the maximal function of $f$,
$$
M(f)(x):=\sup_{r>0}\frac{1}{|B_x(r)|}\int_{B_x(r)} |f(y)|dy,\quad\text{$x\in\R^N$.}
$$
As a consequence of Lemma~\ref{lem-maximal-BN}, we have 
\begin{corollary}\label{cor-maximal} Let $f \in L^1_{\loc}(\R^N)$ and $\rho$ be a nonnegative radial  function such that 
\begin{equation}\label{rhon-cor}
\int_0^\infty \rho(r) r^{N-1} \, dr = 1. 
\end{equation}
Then, for a.e. $x \in \R^N$, 
$$
\int_{B_x(r)} \int_0^1 |f\big(t (y-x) + x \big)| \rho (|y-x|)  \, dt \, d y \le C_N M(f)(x). 
$$
\end{corollary}

\begin{proof} Using polar coordinates, we have 
\begin{equation*}
\int_{B_x(r)} \int_0^1 |f\big(t (y-x) + x \big)| \rho (|y-x|)  \, dt \, d y = \int_{0}^r \int_{\mathbb{S}^{N-1}} \int_0^1 |f\big(x + t s \sigma  \big)| s^{N-1}\rho (s)  \, dt \, d \sigma \,d s. 
\end{equation*}
Applying Lemma~\ref{lem-maximal-BN}, we obtain, for a.e. $x \in \R^N$
$$
 \int_{\mathbb{S}^{N-1}} \int_0^1 |f\big(x + t s \sigma  \big)| \, dt \, d \sigma  \le C_N M(f)(x).  
$$
It follows from \eqref{rhon-cor} that, for a.e. $x \in \R^N$, 
\begin{equation*}
\int_{B_x(r)} \int_0^1 |f\big(t (y-x) + x \big)| \rho (|y-x|)  \, dt \, d y \le C_N M(f)(x), 
\end{equation*}
which is the conclusion. 
\end{proof}

We are ready to give the proof of Proposition~\ref{pro-BBM}.

\begin{proof}[Proof of Proposition~\ref{pro-BBM}] 
We first establish that, for a.e. $x \in \R^N$, 
\begin{equation}\label{pro-BBM-1}
| D_n(u, x) |  \le 
C \Big( M (|\nabla u - \i A u|^2)(x) +   M (|u|^2)(x) \Big) + m \int_{\R^N \setminus B_x(1)} |u(y)|^2\, dy,  
\end{equation} 
where 
$$
m: =2  \sup_{t > 1} \sup_{n} t^{-2} \rho_n(t).
$$ 
Here and in what follows in this proof, $C$ denotes a positive constant independent of $x$.
Indeed, we have, as in \eqref{L1-part1-2},  for a.e. $x, y \in \R^N$ with  $|y - x| < 1$, 
\begin{align}
\frac{|\Psi_u(x, y) - \Psi_u(x, x)|^2}{|x - y|^2} \le &  2 \int_0^1 \big|\nabla u \big(t (y-x) + x \big)   - \i A \big(t (y-x) + x \big)   u \big(t (y-x) + x \big)  \big|^2 \, dt   \nonumber\\[6pt] 
& + 2 \|\nabla A \|_{L^\infty(\R^N)}^2    \int_0^1 \big|u \big(t (y-x) + x \big) \big|^2 \, dt.  
\end{align}
This implies, for a.e. $x \in \R^N$,  
\begin{align*}
\int_{B_x(1)} &  \frac{|\Psi_u(x, y) - \Psi_u(x, x)|^2}{|x - y|^2} \rho_n(|y-x|) \, dy  \\[6pt] 
\le  &  2 \int_{B_x(1)} \int_0^1 \big|\nabla u \big(t (y-x) + x \big)   - \i A \big(t (y-x) + x \big)   u \big(t (y-x) + x \big)  \big|^2  \rho_n(|y-x|)\, dt \, dy    \nonumber\\[6pt] 
& + 2 \|\nabla A \|_{L^\infty(\R^N)}^2  \int_{B_x(1)} \int_0^1 \big|u \big(t (y-x) + x \big) \big|^2 \rho_n(|y-x|)  \, dt \, dy.  
\end{align*}
Applying Corollary~\ref{cor-maximal}, we have, for a.e. $x \in \R^N$,  
\begin{equation}\label{pro-BBM-2}
\int_{B_x(1)}   \frac{|\Psi_u(x, y) - \Psi_u(x, x)|^2}{|x - y|^2} \rho_n(|y-x|) \, dy \le C M(|\nabla u - \i A u|^2)(x) + C M(|u|^2)(x). 
\end{equation}
On the other hand, we  get 
\begin{multline}\label{pro-BBM-3}
\int_{\R^N \setminus B_x(1)}  \frac{|\Psi_u(x, y) - \Psi_u(x, x)|^2}{|x - y|^2} \rho_n(|y-x|) \, dy  \\[6pt]
\le  2|u(x)|^2 + 2 \int_{\R^N \setminus B_x(1)} |u(y)|^2 \rho_n(|y-x|) |x-y|^{-2} \, dy \\[6pt]
\le  2|u(x)|^2 + m \int_{\R^N \setminus B_x(1)} |u(y)|^2 \, d y.  
\end{multline}
A combination of \eqref{pro-BBM-2} and \eqref{pro-BBM-3} yields \eqref{pro-BBM-1}. 
Set, for $v \in H^1_A(\R^N)$ and $\eps \ge 0$, 
\begin{equation*}
\Omega_\eps (v) := \Big\{x \in \R^N: \limsup_{n \to + \infty} \Big|D_n(v, x) - 2 Q_N |\nabla v(x) - \i A(x) v(x)|^2 \Big| > \eps \Big\}. 
\end{equation*} 
By \eqref{L2-part1}, one has, for $v \in C^2_{c}(\R^N)$ and $\eps \ge 0$,  
\begin{equation*}
|\Omega_\eps(v)| = 0. 
\end{equation*}
Using the theory of maximal functions, see e.g., \cite[Theorem 1 on page 5]{Stein}, we derive from \eqref{pro-BBM-1} that, for any $\eps > 0$ and for any $w \in H^1_A(\R^N)$  with  $ m \int_{\R^N} |w(y)|^2 \, dy  \le \eps / 2$, 
\begin{equation}\label{pro-BBM-4}
|\Omega_\eps (w)|  \le \frac{C}{\eps} \int_{\R^N} \left( |\nabla w (x) - \i A(x) w(x) \Big|^2 + |w(x)|^2 \right)\, dx.
\end{equation}
Fix $\eps >0$ and let $v \in C^2_{c}(\R^N)$ with $\max\{1, m \} \|v - u \|_{H^1_A(\R^N)} \le \eps/2$. We derive from \eqref{pro-BBM-4} that  
\begin{equation*}
|\Omega_\eps(u)| \leq |\Omega_\eps(u - v)| \le \frac{C}{\eps} \| v - u \|_{H^1_A(\R^N)}^2 \le C \eps. 
\end{equation*} 
Since $\eps > 0$ is arbitrary, one reaches the conclusion that $|\Omega_0(u)| = 0$. The proof is complete. 
\end{proof}

We next discuss the corresponding results related to Theorem~\ref{main}. 
Given $u \in L^1_{{\rm loc}}(\R^N)$, set, for $x \in \R^N$,  
	$$
	J_\delta(u, x) =  \int_{\{|\Psi_u(x,y)-\Psi_u(x,x)|>\delta\}}\frac{\delta^2}{|x-y|^{N+2}}dy. 
	$$
We have

\begin{proposition}\label{pro-main} Let $A:\R^N\to\R^N$ be Lipschitz and let $u\in H^1_A(\R^N)$.
	We have 
	\begin{equation}\label{pro-conclusion1}
	\lim_{\delta \searrow 0} J_\delta (u, x)= Q_N|\nabla u(x) - \i A(x) u(x)|^2, \quad\mbox{ for a.e. } x \in \R^N  
	\end{equation}
	and 
	\begin{equation}\label{pro-conclusion2}
	\lim_{\delta \searrow 0} J_\delta(u, \cdot) = Q_N|\nabla u( \cdot ) - \i A(\cdot ) u(\cdot )|^2,\quad \mbox{ in } L^1(\R^N).   
	\end{equation}
\end{proposition}

\begin{proof}
	For $v \in H^1_A(\R^N)$, set 
	$$
	\M(v, x) = \int_{\mathbb{S}^{N-1}}  \left(|\M_{\sigma}(|\nabla v - \i A v|, x)|^2 + \| \nabla A\|^2_{L^\infty(\R^N)} |\M_{\sigma}(|v|, x)|^2 \right) \, d \sigma,\quad \mbox{ for } x \in \R^N,
	$$
	and denote	
	$$
	\hat J_\delta(u, x) =  \int_{\{|\Psi_u(x,y)-\Psi_u(x,x)|>\delta, \, |y - x| < 1\}}\frac{\delta^2}{|x-y|^{N+2}}dy,\quad \mbox{for $x\in \R^N$}. 
	$$
	We first establish a variant of  \eqref{pro-conclusion1} and \eqref{pro-conclusion2} in which  $J_\delta$ is replaced by $\hat J_\delta$. 
	Using \eqref{key_p}, as in the proof of Lemma~\ref{lem-variant}, we have, for any $v \in H^1_A(\R^N)$, 
	$$
	\hat J_\delta(v, x) \le C_N \M(v, x)\mbox{ for all } \delta > 0. 
	$$
	We derive that,  for $u, u_n \in H^1_A(\R^N)$, and $\eps \in (0, 1)$, 
	\begin{equation}\label{pro-part1}
	\hat J_{\delta}(u, x) - (1-\eps)^{-2} \hat J_{(1-\eps) \delta} (u_n, x)
	\le \eps^{-2} C_N \M(u-u_n, x), 
	\end{equation}
	and 
	\begin{equation}\label{pro-part2}
	(1 - \eps)^2 \hat J_{\delta/ (1 - \eps)}(u_n, x) - \hat J_{ \delta} (u, x) \\[6pt] \le \eps^{-2} C_N \M(u-u_n,x).  
	\end{equation}
	On the other hand, one can check that, as in the proof of Lemma~\ref{formul}, for $u_n \in C^2_{c}(\R^N)$, 
	\begin{equation}\label{pro-part3}
	\lim_{\delta \searrow 0} \hat J_\delta(u_n, x) = Q_N |\nabla u_n(x) - \i A(x) u_n(x)|^2,\quad \mbox{ for } x \in \R^N. 
	\end{equation}
	We derive from \eqref{pro-part1}, \eqref{pro-part2}, and \eqref{pro-part3} that, for $ u \in H^1_A(\R^N)$,  
	\begin{equation}\label{pro-main-part-1}
	\lim_{\delta \searrow 0} \hat J_{\delta}(u, x) = Q_N |\nabla u(x) - \i A(x) u(x)|^2,\quad 
	\mbox{ for  a.e. } x \in \R^N. 
	\end{equation}
	and, we hence obtain,  by the Dominate convergence theorem, 
	\begin{equation}\label{pro-main-part-2}
	\lim_{\delta \searrow 0} \hat J_{\delta}(u, \cdot) = Q_N |\nabla u(\cdot) - \i A(\cdot) u(\cdot)|^2, \quad \mbox{ in } L^1(\R^N),  
	\end{equation}
	since $\M(u, x) \in L^1(\R^N)$. A straightforward computation yields 
	\begin{equation*}
	\lim_{\delta \searrow 0 } \int_{\{ |y - x| \ge 1\}}\frac{\delta^2}{|x-y|^{N+2}} \, dy  = 0. 
	\end{equation*}
	It follows that 
	\begin{equation}\label{pro-main-part-3}
	\lim_{\delta \searrow 0} [\hat J_{\delta}(u, x) - J_{\delta}(u, x)] = 0,\quad
	\mbox{ for  a.e. } x \in \R^N.
	\end{equation}
	We also have, for $w \in C^2_c(\R^N)$,  
	%\footnote{From H.-M.: Some materials are added here.}
	\begin{align*}
&	\lim_{\delta \searrow 0}	\iint_{\{ |\Psi_w(x, y) - \Psi_w(x, x)| > \delta, \;  |y - x| \ge 1\}}\frac{\delta^2}{|x-y|^{N+2}} dx\, dy \\
	&\leq 	\lim_{\delta \searrow 0}\iint_{\{(B_R\times\R^N)\cup(\R^N\times B_R), \; |y - x| \ge 1\}}\frac{\delta^2}{|x-y|^{N+2}} dx\, dy=0 , 
	\end{align*}
where $R>0$ is such that $\supp w \subset B_R$. 
	Using  Lemma~\ref{lem-variant} and the density of $C^2_c(\R^N)$ in $H^1_A(\R^N)$, we derive that, 
	\begin{equation}\label{pro-main-part-4}
	\lim_{\delta \searrow 0} [\hat J_{\delta}(u, \cdot) - J_{\delta}(u, \cdot)] = 0 \mbox{ in } L^1(\R^N). 
	\end{equation}
	The conclusion now follows from \eqref{pro-main-part-1}, \eqref{pro-main-part-2}, \eqref{pro-main-part-3} and \eqref{pro-main-part-4}. 
\end{proof}

%\medskip

\end{document}